\documentclass[11pt]{amsart}
\usepackage{graphicx}

\usepackage{amsmath,amsfonts,amsthm,amssymb,graphics,color, hyperref, bbm}

\newtheorem{theorem}{Theorem}
\newtheorem{prop}{Proposition}

\newtheorem{ex}{Exercise}

\theoremstyle{remark}
\newtheorem{remark}{Remark}
\theoremstyle{definition}

\newcommand{\cV}{\mathcal{V}}
\newcommand{\one}{\mathbbm{1}}

\numberwithin{equation}{section}

\newcommand{\Prob}{\operatorname{Prob}}

\newcommand{\cC}{\mathcal{C}}

\newcommand{\fS}{\mathfrak{S}} 
\newcommand{\R}{\mathbb{R}}

\renewcommand{\epsilon}{\varepsilon}

        \definecolor{pink}{rgb}{1,0,1}

\begin{document}

\title[The level sets of typical games]{The level sets of typical games}  
\author[Julie Rowlett]{Julie Rowlett} \address{Leibniz Universit\"at Hannover \\ Institut f\"ur Analysis} 
\curraddr{Mathematical Sciences, Chalmers University and the University of Gothenburg, 41296, Sweden.\\ \url{http://www.math.chalmers.se/~rowlett/}} \email{julie.rowlett@chalmers.se}

\begin{abstract}
In a non-cooperative game, players do not communicate with each other.  Their only feedback is the payoff they receive resulting from the strategies they execute.  It is important to note that within each level set of the total payoff function the payoff to each player is unchanging, and therefore understanding the structure of these level sets plays a key role in understanding non-cooperative games.  This note, intended for both experts and non-experts, not only introduces non-cooperative game theory but also shows its fundamental connection to real algebraic geometry.  We prove here a general result about the structure of the level sets, which although likely to be known by experts, has interesting implications, including our recent application to provide a new mathematical explanation for the ``paradox of the plankton.''  We hope to encourage communication between these interrelated areas and stimulate further work in similar directions.  
\end{abstract}


\maketitle   
\section{Introduction} 
What is a non-cooperative game?  You probably know the Rock-Paper-Scissors (RPS) game.  Let's play:  ready, set, go!  What did you choose?  I chose rock.  If you chose scissors, then I win, because rock crushes scissors.  If you chose paper, then you win, because paper covers rock.  If you also chose rock, then it's a draw.  This is an example of a two-player game in which each player has the same three pure strategies:  rock, paper, and scissors.  In order for this game to have any significance, we ought to define a \em payoff function.  \em  For example, we could say the winner receives \$1, which the loser must pay to the winner, so the loser's payoff is -\$1.  If it's a draw, then neither of us receives anything, so the payoffs are both 0.  

This is an example of what we'll call a \em discrete game.  \em  A more general notion is a continuous game, also called a mixed game, which we will simply call a \em game.  \em  Rather than just thinking about playing the game once, we think of the game being repeated an arbitrary or possibly infinite number of times.  Instead of deciding upon one of rock, paper, or scissors, we decide upon a probability distribution which is a list of three numbers corresponding to the probabilities of drawing rock, paper, or scissors.  The sum of these three numbers is 1, because we assume that we must draw something.  One example is $(1/3, 1/3, 1/3)$, which means the probabilities of drawing rock, paper, or scissors are equal to 1/3.  If you only want to draw rock, then your probability distribution would be $(1, 0, 0)$.  We can use these to compute our \em expected payoffs; \em these are known as \em expected values \em in probability theory.  The expected payoff is the sum of the probabilities of each possible outcome multiplied with the payoff according to that outcome.  For me, there is 1/3 chance I will win, which happens if I draw paper; there is a 1/3 chance I will lose, which happens if I draw scissors; and there is a 1/3 chance we will have a draw, which happens if I also draw rock.  Summing these probabilities multiplied with the corresponding payoffs, my expected payoff is 
$$\$ 1* \frac{1}{3} + -\$ 1 * \frac{1}{3} + 0 * \frac{1}{3} = 0.$$
 In your case, there is also a 1/3 chance you will win, which happens if I draw scissors; and there is a 1/3 chance you will lose, which happens if I draw paper; and there is a 1/3 chance we will have a draw, which happens if I draw rock.  So, your expected payoff is also  
 $$\$ 1 * \frac{1}{3} + - \$1 * \frac{1}{3} + 0 * \frac{1}{3} = 0.$$
 More generally, if my probability distribution is $(a, b, c),$ and yours is $(x,y,z)$, corresponding to probabilities of executing rock, paper, or scissors, respectively, then we compute my expected payoff is 
$$-ay + az +bx -bz + cy - cx.$$

\begin{ex} 
What is your expected payoff?  
\end{ex} 

For those of you with some background in game theory, you know that RPS is an example of a two-player, symmetric, zero-sum game which can be given in normal form.  We will see how to express RPS in normal form in \S 2.  The general field of game theory is enormous and has connections to many areas of mathematics, including geometry and analysis.  It would be quite bold to claim to present an exhaustive survey of games and game theory.  Instead, I would like to present an introduction to games which are especially appealing to geometric analysts.  A \em non-cooperative (continuous) game, \em which we will simply call a game, is canonically identified with a \em payoff function, \em from $\R^N$ to $\R^n$, where $n$ is the number of players, and $N$ is the total number of pure strategies summed over all players.  This function will be assumed to be consistent with the definition of expected payoff and expected value.  

\begin{ex} 
What are $n$ and $N$ in the RPS game?
\end{ex} 

For mathematicians in the fields of analysis and geometry, these games are appealing because we can prove theorems about them using the tools of geometry and analysis.  A perfect example is the Nobel-prize winning Nash Equilibrium Theorem.  To state this, we need a few definitions.  

In an $n$-player game, each player has some number of pure strategies, like rock, paper, and scissors.  For more general games it is possible that, unlike in RPS, different players have different sets of pure strategies.  We will use $m_i$ to denote the number of pure strategies the $i^{th}$ player has.  A \em strategy \em for the $i^{th}$ player is a list of $m_i$ non-negative numbers which sum to 1.  These correspond to the probabilities of executing each pure strategy.  Note that these are sometimes called mixed strategies, but since they include the pure strategies, we will simply call them strategies. We can identify each pure strategy with a unit vector in $\R^{m_i}$, because each pure strategy means doing that strategy with probability 1, and the others with probability 0.  So, for instance in RPS, we could identify $(1,0,0) \in \R^3$ with rock, $(0,1,0) \in \R^3$ with paper, and $(0,0,1) \in \R^3$ with scissors.  

The set of strategies for the $i^{th}$ player is the set of all $m_i$-tuples $(c_1, c_2, \ldots, c_{m_i})$ such that 
\begin{equation} \label{mixed} 0 \leq c_j \leq 1, \quad j=1, 2, \ldots, m_i, \quad \sum_{j=1} ^{m_i} c_j = 1.  \end{equation} 
This is nice for geometers because we can geometrically represent the set of strategies for the $i^{th}$ player as the convex hull of the standard unit vectors $\{e_1, e_2, \ldots, e_{m_i}\}$ in $\R^{m_i}$.  The pure strategies are the vertices of this convex set.  We will represent this set by $\fS_i$ and the total strategy space for all players, 
$$\fS = \prod_{i=1} ^n \fS_i.$$
The total strategy space is the product of each of the strategy spaces, so we can view this as a subset of $\R^{N}$, where 
$$N := \sum_{i=1} ^n m_i.$$

The game is represented by $n$ payoff functions which give the expected payoff to each player determined by the strategies across all players  
$$\wp_i : \fS \to \R, \quad i=1, \ldots, n.$$ 

\begin{ex}
Prove that, in order for the payoff function to correspond to the expected value given by the probability distributions over pure strategies, each player's payoff function must be linear in the strategy of that player.  
\end{ex}  

This means that if all other players' strategies are fixed, then each $\wp_i : \fS_i \to \R$ is a linear function.  

The (total) payoff function is 
$$\wp : \fS \to \R^n, \quad \wp = (\wp_1, \wp_2, \ldots, \wp_n).$$
Each component function $\wp_i$ of the total payoff function depends on the strategies of \em all \em players.  Although each function $\wp_i$ is a linear function on $\fS_i$ alone, it need \em not \em in general be simultaneously linear in the strategies of the other players.  For example, the $i^{th}$ player's payoff could depend on the $j^{th}$ player's strategy in a \em non-linear \em way.  We will see more about this in \S 3.  

Now, let us introduce the last bit of notation necessary to state Nash's celebrated theorem.  For $s \in \fS$ and $\sigma \in \fS_i$ let $(s; i; \sigma)$ be the strategy in which the $i^{th}$ player's strategy is replaced by $\sigma$, and all other players' strategies are given by $s$.  An \em equilibrium strategy, \em which is also called an \em equilibrium point, \em is $s \in \fS$ such that 
$$\wp_i (s) \geq \wp_i (s; i; \sigma) \quad \forall \sigma \in \fS_i, \quad \forall i = 1, 2, \ldots, n.$$
This means that no player can increase his payoff by changing his strategy if the strategies of the other players remain fixed.  

\begin{theorem}[Nash]  \label{nash} 
There exists at least one equilibrium strategy in $\fS$.  
\end{theorem} 

The proof is a clever application of the Kakutani Fixed Point Theorem.  Nash defined a function which has a fixed point precisely at an equilibrium point.  He then used the continuity of the total payoff function and the  Fixed Point Theorem to prove that this cleverly defined function must have at least one fixed point.   

In the spirit of Nash's theorem, one can apply geometric analysis to prove a characterization of the level sets of the total payoff function for most games.  We begin in \S 2 with an example from popular culture and a preliminary result based on linear algebra.  We will see in \S 3 that continuity of the payoff function and further properties follow from its definition and use these properties to prove the main theorem.  This result is already recognized by game theorists; see for example \cite{book}.  Nonetheless the proof is instructive for readers learning the theory of non-cooperative games and combines elements of analysis, geometry, geometric measure theory, and algebraic geometry, yet deep knowledge of these areas is not required.  Consequently, we hope the reader also finds the result and its proof interesting.  Although this theorem does not appear to be new, we have made a novel application in biology to the ``paradox of the plankton,'' which is described in \S 4.    
\section*{Acknowledgements} I wish to express sincere gratitude to Susanne Menden-Deuer, Sylvie Paycha, Ulrich Menne, Leobardo Rosales, Jarod Alper, David Trotman, and Wolfgang Ebeling for helpful discussions and suggestions.  The comments of the anonymous referees are also sincerely appreciated and contributed positively to the quality of this note.  As a geometric analyst claiming to be neither an expert in game theory nor in real algebraic geometry, I would like to humbly request that readers consider this an acknowledgement to any references which they feel should have been included and assure that any failure to do so was not intentional!  Finally, I am grateful for the support of the Max Planck Institut f\"ur Mathematik in Bonn, the Georg-August Universit\"at G\"ottingen, the Leibniz Universit\"at Hannover, and the Australian National University.
\section{Preliminaries} 
 In the film, ``A Beautiful Mind,'' based on the life and work of John Nash \cite{nash}, there is a scene which purportedly depicts Theorem \ref{nash}.   
 
 \subsection{A beautiful mind} 
 In this scene Nash is together with a group of male colleagues at a bar as a group of women enters.  One woman is depicted as being thought of as the most attractive to the men, whereas the other women are depicted as being considered only of average attractiveness to the men.  In a flash of insight, Nash's character apparently realizes that he can apply the mathematics he has been studying to determine the best course of action for the men:  he imagines each of the men approaching a different, averagely attractive woman and leaving with her, whereas the most attractive woman is left alone.  At this point Nash hurriedly leaves the bar to work on his new insight.  

The situation is depicted as a competition between the men, where each man decides without communicating with the others which woman he will court.  This corresponds to a non-cooperative game.  For simplicity, let's assume there are 2 men, denoted by man 1 and man 2, and 3 women, denoted by ``M'' (for most attractive) and ``A'' (for averagely attractive).  Each man has two pure strategies:  M which corresponds to courting the most attractive woman, and A which corresponds to courting one of the averagely attractive women.  The \em normal form \em of the game is the following.  

\bigskip
\begin{center} 
\begin{tabular}{|l|c|c|} 
 \hline & M & A \\ \hline M & (0, 0) & (1, -1) \\ \hline A & (-1, 1) & (0, 0) \\ \hline 
\end{tabular}
\end{center} 
 \bigskip 

This is also known as a \em payoff matrix, \em since it lists the payoffs to each player according to the corresponding strategies.  This is an example of a two-player, symmetric, zero-sum game, with one dominant (winning) strategy.  The winning strategy is successfully courting the most attractive woman.  The interpretation of the payoffs is that if both men do strategy M, then they are both unsuccessful.  This means that neither man has won, so they each receive a neutral payoff, 0.  Similarly, if both men do strategy A, and presumably each court a different woman, then they are both successful, but since neither man has won, they each receive a neutral payoff.  In the last case one man does strategy M while the other does strategy A, and so the man doing M has won and receives a payoff of 1, whereas the other man can be seen as the loser and receives a payoff of -1.  

\begin{ex} Represent the rock-paper-scissors game in normal form. 
\end{ex} 

If the probability that man 1 does M is $x$, and the probability that man 2 does M is $y$, then the payoff functions are 
$$\wp_1 (x,y) = x-y, \quad \wp_2 (x,y) = y-x.$$
\begin{ex} Show that the unique equilibrium strategy is $x=y=1$. 
\end{ex} 

The interpretation of the equilibrium strategy is that both men should with probability 1 attempt to court the most attractive woman.  This contradicts the film which indicates that the equilibrium strategy ought to be $x=y=0$.   

One possible explanation is that if indeed the above model was used to determine the best strategy for the characters in the film, perhaps the filmmakers understood that the payoff according to the equilibrium strategy is $(0,0)$.  There is precisely one strategy contained in the level set $\wp^{-1} (0,0)$ $=\{ 0 \leq x = y \leq 1\}$ for which each man pairs up with a woman with probability 1, and that is the strategy $(x=y=0)$ given in the film. While not the equilibrium strategy, it is just as good in the sense that the payoff is identical to the payoff according to the equilibrium strategy, so perhaps this is the reason it is considered to be the best.   

\begin{figure} \includegraphics[width=150pt]{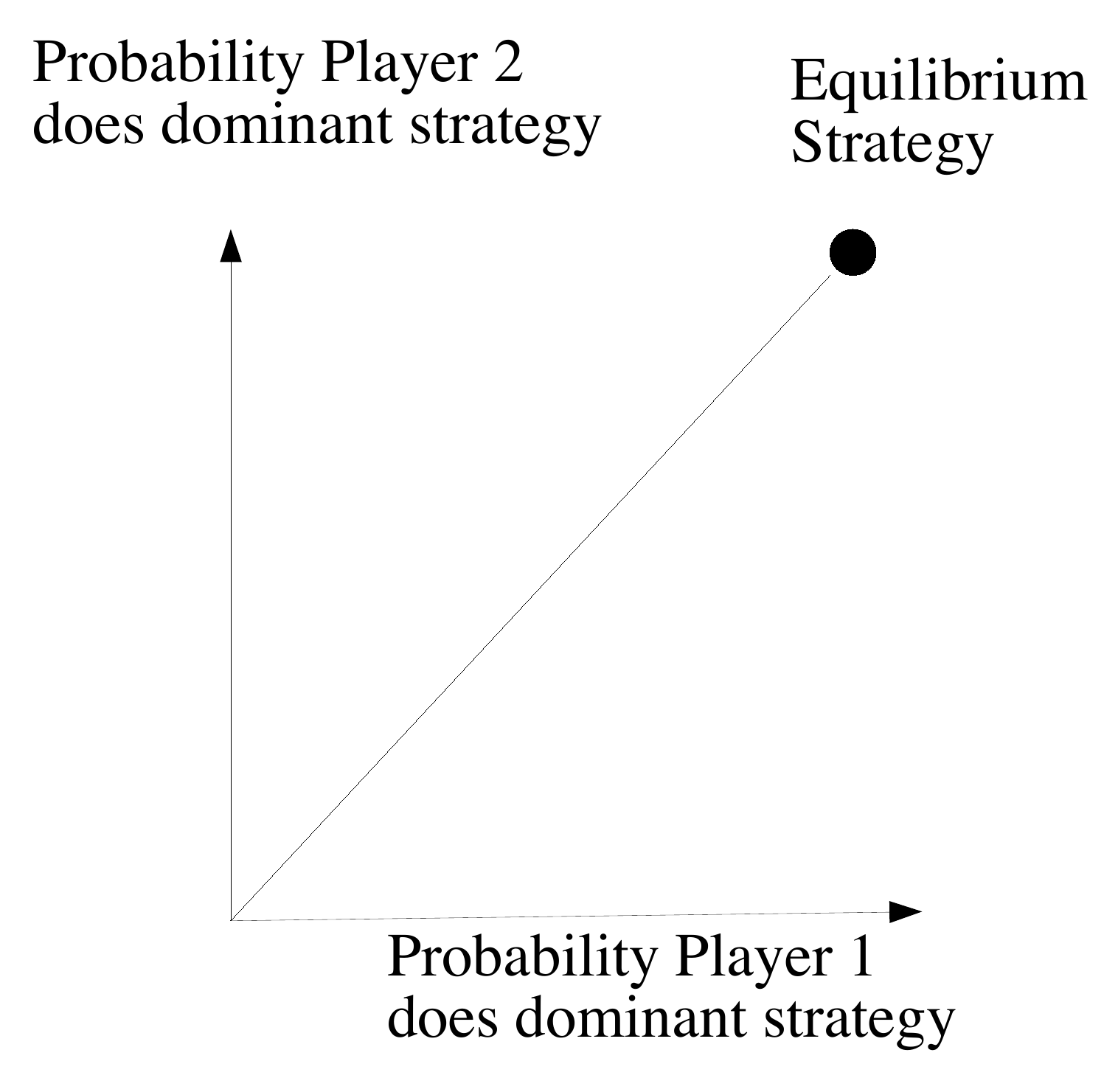}  \caption{For a two-player, symmetric, zero-sum game with one dominant strategy, the level sets of the total payoff function are line segments. This depicts the general idea of Theorem \ref{lots}; the level sets of most games almost always have positive $j$-dimensional Hausdorff measure for some $j \geq 1$.}  \label{fig-peqs}
\end{figure} 
   
It may however seem more natural to define the game with a different payoff matrix.  

\begin{ex} Is it possible to define a game in normal form such that the equilibrium strategy is consistent with the film? 
\end{ex} 
For the solution to this exercise, see \cite{love}.     

\subsection{Linear payoff functions}
While it is always true that the $i^{th}$ player's payoff function depends linearly on his own strategy, it need not be the case that his payoff function also depends linearly on the \em other \em players' strategies.  If, however, this is the case, then we can prove the following result.  This is a good warm-up for the more general characterization of the level sets of payoff functions for typical games.   

\begin{prop} \label{lin}  For an $n$-player game, assume that each player has at least two pure strategies and at least one of the following holds.  
\begin{enumerate}
\item At least one player has three or more pure strategies. 
\item The game is zero-sum.
\end{enumerate}
Let $N$ denote the total number of pure strategies over all players as above.  If only (1) holds, let $k=N-2n$.  If (2) holds, let $k = N-2n +1$.  If the payoff functions are all linear functions in the strategies of all players, then the level sets of the total payoff function are affine linear subsets of dimension $j \geq k$, where $j$ is given in the proof below.  
\end{prop} 

Recall that a game is \em zero-sum \em precisely when, for each $s \in \fS$, 
\begin{equation} \label{zero-sum} \sum_{i=1} ^n \wp_i (s) = 0. \end{equation} 
Zero-sum games imply that gains by some players are met by equal losses to other players and can therefore be used to model competition for limited resources.  

\begin{proof}
Do you remember the Rank-Nullity Theorem from linear algebra?  This theorem states that for an $m \times n$ matrix, the sum of the dimension of the column space (rank) together with the dimension of the kernel (nullity) is equal to $n$.  The idea is that if one uses Gauss-Jordan elimination to put the matrix in row-reduced echelon form, then each column is either a pivot column or not.  The number of pivot columns is the rank, and the number of non-pivot columns is the nullity.  This number must sum up to the total number of columns, and that is $n$.  So, what does this mean if we have a matrix which is longer than it is tall, so that $n > m$?  Since the column space is the dimension of the space spanned by the columns, and each column is an element of $\R^m$, this dimension is at most $m$.  Since $n>m$, the nullity must be at least $n-m$.  It turns out that the payoff function $\wp$ for most games can be canonically identified with a map from a higher dimensional Euclidean space to a lower dimensional Euclidean space.  

\begin{ex}  Show that the strategy space $\fS$ is an $N-n$ dimensional subset of $\R^N$.  Using \eqref{zero-sum} if the game is zero-sum, show that the payoff function is canonically identified with a map from $\R^{N-n}$ to $\R^{N-n - k}$.  
\end{ex} 

Since the payoff function can be represented by an affine linear function from $\R^{N-n}$ to $\R^{N-n-k}$, there exists an $(N-n-k) \times (N-n)$ matrix $M$ and a $(N-n-k) \times 1$ vector $b$ such that $\wp(s) = Ms + b$.  The level sets of $\wp$ are translations of the kernel of $M$, and since $M$ is $k$-columns wider than it is tall, by the Rank-Nullity Theorem the dimension of the kernel of $M$ is $j \geq k$.  
\end{proof} 
 
This proposition may be helpful in familiarizing readers with games and payoff functions, since it relies only on the definitions and linear algebra.  What is more interesting is that this result can be generalized to payoff functions which are \em not \em linear in the strategies of \em all \em players.  The assumptions (1) and (2) above are satisfied by a ``typical game,'' because if a player has only one pure strategy, then he cannot affect the outcome of play, so his role is trivial.  Moreover, many games have at least three pure strategies per player and/or are zero-sum.

\section{Main result} 
While John Nash was a graduate student at Princeton in 1950, he proved the existence of equilibrium strategies for non-cooperative games \cite{nash}.  In 1952, he published \em Real algebraic manifolds \em and proved that two real algebraic manifolds are equivalent if and only if they are analytically homeomorphic \cite{nash2}.  He then proceeded in 1954--1956 to study the imbedding problem for Riemannian manifolds \cite{nash3}, \cite{nash4}.  That work involved what are now known as Nash functions and Nash manifolds; the payoff functions considered here are examples of Nash functions.  Based on his work in game theory, differential geometry, and algebraic geometry, we can be pretty sure that Nash was the first to recognize this result and therefore acknowledge it to him.\footnote{The author would like to note that this result and its proof, although implicitly or explicitly known by experts, was obtained independently.}  

\begin{theorem}[Nash] \label{lots} For an $n$-player game, assume that each player has at least two pure strategies, and at least one of the following holds.  
\begin{enumerate}
\item At least one player has three or more pure strategies. 
\item The game is zero-sum.
\end{enumerate}
The image $\wp(\fS)$ is then a $k$-dimensional semialgebraic set for some $k \leq n$ in case only (1) holds or $k \leq n-1$ in case (2) holds.  For almost every $y \in \wp(\fS)$ with respect to $k$-dimensional Hausdorff measure, the level set $\wp^{-1} (y)$ has positive (or infinite) $N-n-k$-dimensional Hausdorff measure, noting that $N-n-k \geq 1$.  
\end{theorem} 

\begin{proof}
The main idea is to put a game which satisfies these hypotheses in normal form.  The strategy space 
$$\fS \cong \prod_{i=1} ^n \fS_i, \quad \fS_i \cong \left\{ \sum_{j=1} ^{m_i} c_j e_j \in \R^{m_i} : 0 \leq c_j \leq 1, \, \sum_{j=1} ^{m_i} c_j = 1\right\},$$
where $e_j$ are used to denote the standard unit vectors in Euclidean space.  The normal form of a game lists the payoffs to all players corresponding to all possible combinations of pure strategies.  These combinations of pure strategies geometrically correspond to the vertices of $\fS$.  We denote this set by $\cV$ and use binary expansions of integers to represent the elements of $\cV$.  There are 
$$M:= \prod_{i=1} ^n m_i$$ 
elements of $\cV$.  Each element is of the form 
$$v_x = \sum_{j=1} ^N x_j e_j, \, x_j \in \{0,1\} \, \forall j, \quad x := \sum_{j=1} ^N x_j 2^j \in \{2, 4, \ldots, 2^{N+1} -2\}.$$
So we see that each vertex $v_x$ corresponds to a unique $x$, because binary expansions are unique (see Chapter 6 of \cite{moi}).  There is one further restriction on the vertices:  each player executes \em one \em pure strategy at a vertex $v_x$.  Mathematically we can express this using the orthogonal projections 
$$\phi_i : \fS \to \fS_i,$$
together with 
$$\one_i := \sum_{j=1} ^{m_i} e_j,$$
so that 
$$\phi_i (v_x) \cdot \one_i = 1, \quad \forall i=1, \ldots, n.$$
The normal form for such a game would in this generality be a rather large matrix.  Each player requires one dimension, so the matrix is $n$-dimensional.  Along the $i^{th}$ dimension there are $m_i$ slots, corresponding to each of the $m_i$ possible pure strategies for the $i^{th}$ player.  In an entry of this matrix, we list the payoffs to each player for the corresponding list of pure strategies.  Once we know all these payoffs, then just like the RPS game, we can write the payoff for any mixed strategy, because this must be consistent with the expected value.  So, for a strategy 
$$s = \sum_{j=1} ^N c_j e_j \in \fS, \quad \wp_i (s) = \sum_{v_x \in \cV} \Prob_{v_x} (s) \wp_i (v_x).$$
Above, we used $\Prob_{v_x} (s)$ to denote the probability according to $s$ of the combination of pure strategies in $v_x$.  This is the product of the probabilities according to $s$ of each pure strategy in $v_x$, 
$$\Prob_{v_x} (s) = \prod_{j=1} ^N x_j c_j.$$
So, what does this mean?  The payoff functions 
$$\wp_i (s) = \sum_{v_x \in \cV} \left( \prod_{j=1} ^N x_j c_j \right) \wp_i (v_x).$$
The important observation is that this is a polynomial function in the variables $\{c_j\}_{j=1} ^N$.  The total payoff function is therefore a real polynomial function from $\fS \subset \R^N \to \R^n$.  Since $\fS$ is defined by a finite set of inequalities and linear equations, it is by definition a semialgebraic set (see Definition 2.1.4 on p. 24 of \cite{rag}).  By the Tarski-Seidenberg Theorem (see p. 28--29 of \cite{rag}), $\wp(\fS)$ is also semialgebraic set.  Such a set has the structure of a stratified space, which is a disjoint union of a finite number of smooth manifolds (strata) which are themselves semialgebraic sets, and such that this stratification can be taken to satisfy the Whitney conditions \cite{rev}.  In this case, since the payoff function is continuous, and $\fS$ is compact, the image $\wp(\fS)$ is compact, and so this semialgebraic set is triangulable and is semialgebraically isomorphic to a finite polyhedron \cite{rev}.  It has some dimension $k \leq n$.  Note that if the game is zero-sum, then 
$$\wp_n (s) = 1 - \sum_{i=1} ^{n-1} \wp_i (s),$$
which implies that $\wp(\fS)$ has dimension $k \leq n-1$.  The level sets, $\wp^{-1} (\wp(s))$ for $\wp(s) \in \wp(\fS)$ are known in this setting as fibers.  By Theorem 9.3.2 and Corollary 9.3.3 on p. 221--224 of \cite{rag}, we can decompose $\wp(\fS)$ as the union of semialgebraic sets 
$$\wp(\fS) = \bigcup_{l=0} ^L T_l, \quad \dim(T_0) = k, \quad \dim(T_l) < k, \quad \forall l\geq1,$$
such that each $T_l$ is closed for $l \geq 1$, and $\wp$ has a semialgebraic trivialization over each $T_l$.  This means that for each $l$, the fibers $f^{-1} (y)$ have dimension $d_l$ for all $y \in T_l$.  By removing the lower dimensional strata, $T_0$ is an open $k$-dimensional semialgebraic set.  By Proposition 2.38 on p. 71 of \cite{algor}, $\wp^{-1} (T_0)$ is an open semialgebraic set (openness follows since $\wp$ is a polynomial and therefore continuous).  By the Semialgebraic Sard Theorem (see Theorem 9.6.2 on p. 235 of \cite{rag}) the set of critical values in $T_0$ has dimension strictly smaller than $k$.  At a regular (not critical) point, the derivative matrix $D \wp(s)$ has rank equal to $k$, and the level set $\wp^{-1} (\wp(s))$ is an $N-n-k$ dimensional submanifold of $\fS$ (recall that $\fS$ is $N-n$ dimensional).  Since the dimension of the fibers over $T_0$ are all the same, this means that the level sets $\wp^{-1} (y)$ have dimension $N-n-k$ for all $y \in T_0$.  
Furthermore the sets $\wp(\fS)$ and $T_0$ differ by a set of zero $k$-dimensional Hausdorff measure.  Consequently, for almost all (with respect to $k$-dimensional Hausdorff measure) $y \in \wp(\fS)$, the level set $\wp^{-1} (y)$ is an $N-n-k$ dimensional submanifold and therefore has positive (or infinite) $N-n-k$ dimensional Hausdorff measure. 
\end{proof}

\begin{remark}
One could likely say more about the structure of the level sets (fibers) using the tools of real algebraic geometry; references include \cite{rag}, \cite{rag2}, \cite{algor}, \cite{nashmf}, \cite{gory}.  For our biological application discussed in \S 4, the above Theorem was sufficient. 
\end{remark}

\subsection{Bibliographical Note} 
Being rather new to game theory it was natural to search the literature for results concerning the level sets of payoff functions.  In \cite{throats}, the level sets of the value (payoff) function for linear differential games of fixed terminal time with a convex payoff function were numerically investigated.  For a linear pursuit-evasion game with two pursuers and one evaders, the level sets of the value function were numerically studied in \cite{evade}.   For zero-sum games, \cite{flows} studied a certain Hamiltonian flow which can be used to study the best response dynamic in two-person games, and showed that under certain assumptions the level sets of the associated Hamiltonian function are topological spheres.  Further examples of the study of the level sets of the payoff functions for specific games include \cite{kurt}, \cite{pats}, \cite{pats2}, and \cite{sauk}.  Investigating connections between real algebraic geometry and game theory led to Neyman's work including \cite{book}.  It appears that many results in game theory tend to be more computational whereas the results in real algebraic geometry tend to be more theoretical.  We hope to encourage further communication between game theorists and real algebraic geometers.  
  
 \section{Applications}  
The structure of the level sets of the payoff functions has a novel application to biology by providing a new and rigorous solution to the long-standing ``paradox of the plankton'' in \cite{mdr}.  

\subsection{Biodiversity of Micro-Organisms} 
The ``paradox of the plankton'' coined by Hutchinson in 1961 \cite{hutch} is the observation that the number of co-existing plankton species appears to contradict the explicable number based on competition theory \cite{har60}, \cite{gause}.  The number of co-existing species is orders of magnitude larger than expected, based on competition theories and predictions which yield reasonably accurate numbers for macro-organisms.  There have been numerous explanations proposed by biologists, but a mathematical theory consistent with all these explanations which is based on a biological factor subject to natural selection and is not in contradiction to competition theory appears to have been missing.  In \cite{mdr}, we realized that Theorem \ref{lots} has implications for a game modeling competition of plankton organisms which may resolve the paradox.  

How can we use a game to model competition of plankton organisms?  Plankton reproduce asexually and are genetically identical within a species.  Justified by this clonal nature we define a ``player'' as consisting of many individuals belonging to one species.  The survival of the species is a cumulative function of the survival of its individuals.  Due to the asexual reproduction, success in competition among microbes can be identified with population increase or decrease, which corresponds to positive or negative payoff.  The strategies for each player (=species) are probability distributions across the various behaviors of which that species is capable.  Each of these probabilities is the probability that a randomly selected individual organism does the corresponding behavior (like swim up, for example).  In this way, we can use a game to model competition between plankton species.  

What is the connection to the structure of the level sets of typical games?  In a broad sense evolution can be described as a feedback loop; we refer readers interested in evolutionary game dynamics to \cite{ev} and the references therein.  This means that within a level set of a game modeling competition, the feedback to all species is identical.  There is absolutely no difference.  In the generic sense made precise in Theorem \ref{lots}, the level sets of typical games are typically \em large. \em  How do the hypotheses of the theorem fit with plankton ecology?  The hypotheses mean that all species are capable of at least two \em different behaviors.  \em  This corresponds to \em individual variability \em which seems to be the underlying mechanism supporting the large plankton biodiversity and may explain the unexpectedly large biodiversity of other microbes as well.  The assumptions (1) and (2) mean that either species possess further variability and/or are competing for limited resources which also appears to be the case.  

Although plankton individuals are genetic clones within a species, they exhibit significant variability among individuals; see \S 1 of \cite{mdr} and the references cited therein.  This individual variability is inherent to a species and is subject to natural selection and consequently to the evolutionary feedback loop \cite{falk04}.  We propose that it is this individual variability which is driving the large biodiversity.  Mathematically, by Theorem 2, this variability implies that the level sets of the payoff function are large.  These large level sets correspond to the large variety of strategies which are all ``equally good,'' in the sense that they produce identical feedback.  The various strategies in the level sets may characterize different species and correspond to the large number of species which may co-exist, which we describe as ``many different ways to stay in the game" \cite{mdr}.  

Our theory may be thought of as ``survival of the cumulatively fit'' rather than ``survival of the fittest" and is applicable to micro-organisms which re-produce asexually.  Defining a player as consisting of several organisms belonging to one species, while reasonable for micro-organisms which reproduce asexually, no longer makes sense for larger macro-organisms which do not reproduce asexually, because the death of an individual implies the loss of that individual's unique genome.  Consequently, our theory does not contradict competition theories or predictions of species abundance for macro-organisms.  

\subsubsection{Compatibility} It may seem counter-intuitive to apply non-cooperative game theory to evaluate a relationship, but many everyday decisions are made quickly according to self-interest, without cooperative discussion.   In \cite{love} we used non-cooperative game theory to design a new type of compatibility test to measure the balance and overall happiness of two people in a relationship.  Our test may be customized to analyze the overall balance and satisfaction in any relationship between two people, romantic or otherwise; this is discussed in \cite{love}.  If you are in a relationship, we challenge you to take this test!  

\subsection{Concluding Remarks} 
In situations modeled by non-cooperative games players do not communicate; the only feedback they experience is their payoff.  This means that not only equilibrium strategies but also the structure of the level sets of the payoff function are important to understand.  Some readers may be of the opinion that only seasoned experts ought to write about a certain topic.  However, approaching a field from a different perspective may at times be helpful, and so I hope that this note written from a geometric analyst's perspective has provided some basic insight to non-cooperative (mixed/continuous) games, and that it may inspire further investigation and collaboration.

{}

\end{document}